\documentclass[11pt,a4paper]{article}
\usepackage[round]{natbib}
\usepackage{amsfonts}
\usepackage{amsmath}
\usepackage{amssymb}
\usepackage{amsthm}
\usepackage[dvips]{graphicx}
\usepackage{pifont}
\usepackage{enumerate}
\theoremstyle{definition}
\newtheorem{thm}{Theorem}
\newtheorem{lemma}{Lemma}
\newtheorem{definition}{Definition}
\newcommand*{\Cdot}{\raisebox{-0.25ex}{\scalebox{1.4}{${\cdot}$}}}
\newcommand{\gx}[2]{g_{#1}^{#2}}
\newcommand{\gy}[2]{\tilde{g}_{#1}^{#2}}%
\newcommand{\gxy}[2]{\mathring{g}_{#1}^{#2}}%
\newcommand{\dd}{\mbox{d}}
\newcommand{\ee}{\mbox{e}}
\newcommand{\E}{{\rm E}}
\begin{document}
\begin{center}
{\Large Simultaneous prediction for independent Poisson processes \\ with different durations} \\
\vspace{20pt}
{\large Fumiyasu Komaki $^{1,2}$} \\
{\small $^{1}$ Department of Mathematical Informatics} \\
{\small Graduate School of Information Science and Technology, the University of Tokyo} \\
{\small 7-3-1 Hongo, Bunkyo-ku, Tokyo 113-8656, JAPAN} \\
{\small komaki@mist.i.u-tokyo.ac.jp} \\
{\small $^{2}$ RIKEN Brain Science Institute} \\
{\small 2-1 Hirosawa, Wako City, Saitama 351-0198, JAPAN}
\end{center}

\vspace{10pt}

\begin{center}
Abstract
\end{center}

Simultaneous predictive densities for independent Poisson observables are investigated.
The observed data and the target variables to be predicted are independently distributed
according to different Poisson distributions parametrized by the same parameter.
The performance of predictive densities is evaluated by the Kullback--Leibler divergence.
A class of prior distributions depending on the objective of prediction
is introduced.
A Bayesian predictive density based on a prior in this class dominates
the Bayesian predictive density based on the Jeffreys prior.

\vspace{0.5cm}

\section{Introduction}

Suppose that $x_i$ $(i=1,\ldots,d)$ are independently distributed
according to the Poisson distribution
with mean $r_i \lambda_i$
and that $y_i$ $(i=1,\ldots,d)$ are independently distributed
according to the Poisson distribution
with mean $s_i \lambda_i$.
Then,
\begin{align}
p(x \mid \lambda)
= \prod^d_{i=1} \frac{(r_i \lambda_i)^{x_i}}{x_i!} \ee^{-r_i \lambda_i},
\label{xdist}
\end{align}
and
\begin{align}
p(y \mid \lambda)
= \prod^d_{i=1} \frac{(s_i \lambda_i)^{y_i}}{y_i!} \ee^{-s_i \lambda_i},
\label{ydist}
\end{align}
where $x = (x_1,\ldots,x_d)$ and $y = (y_1,\ldots,y_d)$.
Here, $\lambda := (\lambda_1, \ldots, \lambda_d)$ is the unknown parameter, and
$r = (r_1,\ldots,r_d)$ and $s = (s_1,\ldots,s_d)$ are known positive constants.
The objective is to construct a predictive density $\hat{p}(y ; x)$ for $y$ by using $x$.

The performance of $\hat{p}(y;x)$ is evaluated by the Kullback--Leibler divergence
\[
 D(p(y \mid \lambda), \hat{p}(y;x)) := \sum_y p(y \mid \lambda) \log \frac{p(y \mid \lambda)}{\hat{p}(y;x)}
\]
from the true density $p(y \mid \lambda)$ to the predictive density $\hat{p}(y;x)$.
The risk function is given by
\[
 \E \Bigl[ D(p(y \mid \lambda), \hat{p}(y;x)) \, \Big| \, \lambda \Bigr]
= \sum_x \sum_y p(x \mid \lambda) p(y \mid \lambda) \log \frac{p(y \mid \lambda)}{\hat{p}(y;x)}.
\]
It is widely recognized that Bayesian predictive densities
\[
 p_\pi(y \mid x) :=
\frac{\int p(y \mid \lambda) p(x \mid \lambda) \pi(\lambda) \dd \lambda}
{\int p(x \mid \lambda) \pi(\lambda) \dd \lambda},
\]
where $\dd \lambda = \dd \lambda_1 \cdots \dd \lambda_d$,
constructed by using a prior $\pi$ perform better than
plug-in densities $p(y \mid \hat{\lambda})$ constructed by replacing the unknown parameter $\lambda$
by an estimate $\hat{\lambda}(x)$.
The choice of $\pi$ becomes important to construct a Bayesian predictive density.

The Jeffreys prior
\begin{align}
\label{jeffreysx}
\pi_\mathrm{J} (\lambda) \dd \lambda_1 \dotsb \dd \lambda_d
\propto \lambda_1^{-\frac{1}{2}} \dotsb \lambda_d^{-\frac{1}{2}} \dd \lambda_1 \dotsb \dd \lambda_d
\end{align}
for $p(x \mid \lambda)$ is proportional to the Jeffreys prior for $p(y \mid \lambda)$ and
the volume element prior $\pi_\mathrm{P} (\lambda)$ with respect to
the predictive metric discussed in section 4.
A natural class of priors including the Jeffreys prior is
\begin{align*}
\pi_\beta (\lambda) \dd \lambda_1 \dotsb \dd \lambda_d
:= \lambda_1^{\beta_1-1} \dotsb \lambda_d^{\beta_d-1} \dd \lambda_1 \dotsb \dd \lambda_d,
\end{align*}
where $\beta_i > 0$ $(i=1,\ldots,d)$.

We introduce a class of priors defined by
\[
\pi_{\alpha,\beta,\gamma} (\lambda) \dd \lambda_1 \dotsb \dd \lambda_d :=
\frac{\lambda^{\beta_1-1}_1 \dotsb \lambda^{\beta_d -1}_d}
{(\lambda_1/\gamma_1 + \dotsb + \lambda_d/\gamma_d)^\alpha}
\dd \lambda_1 \dotsb \dd \lambda_d,
\]
where $0 \leq \alpha \leq \beta_{\Cdot} := \sum_i \beta_i$, $\beta_i > 0$,
and $\gamma_i > 0$ $(i=1,\ldots,d)$.
In the following, a dot as a subscript indicates summation over the corresponding index.
Note that $\pi_{\alpha,\beta,\gamma} \propto \pi_{\alpha,\beta,c\gamma}$,
where $c>0$ and $c\gamma = (c\gamma_1,\ldots,c\gamma_d)$.
The prior $\pi_{\alpha,\beta,\gamma}$ does not depend on $\gamma := (\gamma_1,\ldots,\gamma_d)$ if $\alpha = 0$.
If $\alpha > 0$, $\pi_{\alpha, \beta, \gamma}$ puts more weight on
parameter values close to $0$ than $\pi_\beta$ does.
In this sense, $\pi_{\alpha, \beta, \gamma}$ with $\alpha > 0$ is a shrinkage prior.

There have been several studies for the simple setting
$r_1 = r_2 = \cdots = r_d$ and $s_1 = s_2 = \cdots = s_d$.
Decision theoretic properties of linear estimators
under the Kullback--Leibler loss is studied by \cite{GY1988AS}.
The theory for Bayesian predictive densities for the Poisson model
is a generalization of that for Bayesian estimators under the Kullback--Leibler loss.
A class of priors $\pi_{\alpha,\beta} := \pi_{\alpha,\beta,\gamma}$
with $\gamma_1 = \cdots = \gamma_d = 1$ is introduced
in \citet{Komaki:TheAnnalsOfStatistics:2004}.
It is shown that the risk of the Bayesian predictive density based on $\pi_{\tilde{\alpha},\beta}$
with $\tilde{\alpha} := \beta_{\Cdot} -1$
is smaller than the risk of that based on $\pi_\beta$ if $\beta_{\Cdot} > 1$.
For example, if $d \geq 3$,  there exists a Bayesian predictive density that dominates the Bayesian predictive density
$p_\mathrm{J}(y \mid x)$
based on the Jeffreys prior because $\beta_{\Cdot} = d/2 > 1$.
Here, $p_\pi(y \mid x)$ is said to dominate $p_\mathrm{J}(y \mid x)$
if the risk of $p_\pi(y \mid x)$ is not greater than that of $p_\mathrm{J}(y \mid x)$
for all $\lambda$ and 
the strict inequality holds
for at least one point $\lambda$ in the parameter space.

Bayesian predictive densities based on shrinkage priors are discussed by
\cite{komaki2001shrinkage} and \cite{George:TheAnnalsOfStatistics:2006} for normal models.
See also \cite{George:StatisticalScience:2012} for recent developments of the theory of predictive densities.
In practical applications, it often occurs that observed data $x$ and the target variable $y$ to be predicted
have different distributions parametrized by the same parameter.
Regression models with the same parameter and different explanatory variable values are a typical example.
\cite{Kobayashi:JournalOfMultivariateAnalysis:2008} and \cite{george2008predictive}
showed that shrinkage priors are useful for constructing
Bayesian predictive densities for
normal linear regression models.
\cite{KBA2014} has studied asymptotic theory for general models other than normal models
when $x(i)$ $(i=1,\ldots,N)$ and $y$ have different distributions $p(x \mid \theta)$ and $p(y \mid \theta)$,
respectively, with the same parameter $\theta$.
However, there has been few studies on nonasymptotic theories
of Bayesian predictive densities for non-normal models
when the distributions of $x$ and $y$ are different.

In the present paper, we develop finite sample theory for prediction
when the data $x$ and the target variable $y$ have
different Poisson distributions \eqref{xdist} and \eqref{ydist}, respectively, with the same parameter $\lambda$.
The proposed prior depends not only on $r$ corresponding to the data distribution but also on $s$
corresponding to the objective of prediction.
Thus, we need to abandon the context invariance of the prior, see e.g.~\cite{dawid1983invariant}.
The Bayesian predictive densities studied in the present paper are not represented
by using widely known functions such as gamma or beta functions,
contrary to the simple setting $r_1 = \cdots = r_d$ and $s_1 = \cdots = s_d$ \citep{Komaki:TheAnnalsOfStatistics:2004}.
However, the predictive densities are represented by introducing
a generalization of the Beta function, and the results are proved analytically.

In section 2, 
we formulate the problem as prediction for time-inhomogeneous Poisson processes
and the risk function is represented as an integral with respect to the time.
In section 3, we show that a Bayesian predictive density based on a prior 
in the introduced class $\pi_{\alpha,\beta,\gamma}$ dominates that based on $\pi_\beta$ if $\beta_{\Cdot} > 1$.
The harmonic time $\tau$ for the time-inhomogeneous Poisson processes is introduced to prove the results.
In section 4, we discuss several properties of the proposed prior and the harmonic time $\tau$.

\section{Evaluation of risk}
We formulate the problem as prediction for time-inhomogeneous Poisson processes
and obtain a useful expression of the risk.

Let
$t_i (\tau)$ $(i=1, \dotsb, d)$ be smooth monotonically increasing functions of $\tau \in [0,1]$ satisfying
$t_i (0) = r_i$ and $t_i (1) = r_i + s_i$.
Let $z_i(\tau)$ $(i=1, \dotsb, d)$ be independent time-inhomogeneous Poisson processes with mean
$t_i(\tau) \lambda_i$ and time $\tau$.
Then, the density of $z(\tau)$ is
\[
p(z(\tau) \mid \lambda) 
= \prod^d_{i=1} \frac{\{t_i(\tau) \lambda_i\}^{z_i}}{z_i!} \ee^{-t_i(\tau) \lambda_i},
\]
where $z(\tau) := (z_1(\tau),\ldots,z_d(\tau))$,
and the distributions of $z_i(0)$ and $z_i(1) - z_i(0)$ are identical with
those of $x_i$ and $y_i$, respectively.
Since $z(0)$ and $z(1)-z(0)$ are independent,
prediction of $y$ based on $x$ is equivalent to
prediction of $z(1)-z(0)$ based on $z(0)$.
We identify $x$ and $y$ with $z(0)$ and $z(1)-z(0)$, respectively.

Let $z_\Delta(\tau) := z(\tau+\Delta) - z(\tau)$.
Then, $z_{\Delta=1}(0)$ corresponds to $y$.
The density of $z_\Delta(\tau)$ is
\begin{align*}
p(z_\Delta(\tau) \mid \lambda) =&
\prod^d_{i=1} \frac{[\{t_i(\tau+\Delta) - t_i(\tau)\} \lambda_i]^{(z_\Delta)_i}}{(z_\Delta)_i!}
\ee^{- \{t_i(\tau+\Delta) - t_i(\tau)\} \lambda_i}.
\end{align*}
We designate the prediction of $z_\Delta(\tau)$ in the limit $\Delta \rightarrow 0$ as infinitesimal prediction.

The following lemma represents the risk of the original prediction as an integral of the risk of infinitesimal prediction.

\vspace{0.5cm}

\begin{lemma}\label{smallprediction}
\begin{description}
 \item{~1)~}
Let $\pi(\lambda)$ be a prior density.
Then,
\begin{align}
\frac{\partial}{\partial \Delta} &
\E\Bigl[ D\{p(z_\Delta(\tau) \mid \lambda), p_\pi(z_\Delta(\tau) \mid z(\tau))\} \; \Big| \; \lambda \Bigr]
\bigg|_{\Delta = 0} \notag \\
\label{hyouka-2}
=& \frac{\partial}{\partial \tau} D \{ p(z(\tau) \mid \lambda), p_\pi(z(\tau)) \} \\
\label{hyouka-3}
=& \E \left[
\sum^d_{i=1} \dot{t}_i(\tau)
\biggl\{
\hat{\lambda}^\pi_i(z(\tau),\tau) -  \lambda_i
 - \lambda_i \log \frac{\hat{\lambda}^\pi_i (z(\tau),\tau)}{\lambda_i} \biggr\}
 \; \Bigg| \; \lambda \right],
\end{align}
where
\[
p_\pi (z(\tau)) := \int p(z(\tau) \mid \lambda) \pi (\lambda) \dd \lambda
= \int \prod^d_{i=1} \frac{\{t_i(\tau) \lambda_i \}^{z_i}}{z_i!}
\ee^{-t_i(\tau) \lambda_i} \pi (\lambda) \dd \lambda,
\]
\[
\hat{\lambda}_i^{\pi}(z(\tau),\tau) :=
\frac{\int \lambda_i p(z(\tau) \mid \lambda) \pi (\lambda) \dd \lambda}{\int p(z(\tau) \mid \lambda)
\pi (\lambda) \dd \lambda},
\]
and
\[
\dot{t}_i(\tau) := \frac{\dd}{\dd \tau} t(\tau).
\]
 \item{~2)~}
Let $\pi(\lambda)$ and $\pi'(\lambda)$ be prior densities, and let $p_\pi(y \mid x)$ and $p_{\pi'}(y \mid x)$
be the corresponding Bayesian predictive densities.
Then,
\begin{align}
\E \bigl[ & D(p(y \mid \lambda), p_{\pi'} ( y \mid x)) \, \big| \, \lambda \bigr]
- \E \bigl[ D (p(y \mid \lambda), p_{\pi} (y \mid x)) \, \big| \, \lambda \bigr] \notag \\
= & \int^1_0 \frac{\partial}{\partial \Delta} \E
\big[ D ( p ( z_\Delta(\tau) \mid \lambda ) , p_{\pi'} ( z_\Delta  (\tau) \mid z ( \tau) ) \mid \lambda  \big]
\bigg\vert_{\Delta=0}
 \dd \tau \notag \\
& \label{lemma1-2-1}
- \int^1_0 \frac{\partial}{\partial \Delta} \E 
\big[ D ( p ( z_\Delta(\tau) \mid \lambda), p_\pi ( z_\Delta (\tau) \mid z ( \tau) ) \mid \lambda \big]
\bigg\vert_{\Delta=0} \dd \tau \\
= & \int^1_0
\E \left[ \sum_i \dot{t}_i(\tau) \biggl\{
\hat{\lambda}^{\pi'}_i (z(\tau),\tau)
- \lambda_i - \lambda_i \log \frac{\hat{\lambda}^{\pi'}_i (z(\tau),\tau)}{\lambda_i}
\biggr\}
\; \Bigg| \; \lambda \right]
 \dd \tau \notag \\
& \label{lemma1-2-2}
- \int^1_0
\E \left[ \sum_i \dot{t}_i(\tau) \biggl\{
\hat{\lambda}^\pi_i(z(\tau),\tau) - \lambda_i - \lambda_i \log \frac{\hat{\lambda}^\pi_i (z(\tau),\tau)}{\lambda_i}
\biggr\}
\; \Bigg| \; \lambda \right]
\dd \tau.
\end{align}
\end{description}
\qed
\end{lemma}

\vspace{0.5cm}

Equation \eqref{hyouka-3} shows that infinitesimal Bayesian prediction based on $\pi$
corresponds to the Bayesian estimator $\hat{\lambda}_\pi$.
This fact is a generalization of a result discussed in \cite{Komaki:JournalOfMultivariateAnalysis:2006}
when $r_1=\cdots=r_d$ and $s_1= \cdots =s_d$.
By \eqref{lemma1-2-2}, if
\begin{align*}
\E & \left[ \sum_i \dot{t}_i(\tau) \biggl\{
\hat{\lambda}^{\pi'}_i (z(\tau),\tau) - \hat{\lambda}^{\pi}_i (z(\tau),\tau)
- \log \frac{\hat{\lambda}^{\pi'}_i (z(\tau),\tau)}{\hat{\lambda}^{\pi}_i (z(\tau),\tau)}
\biggr\}
\; \Bigg| \; \lambda \right]
\end{align*}
is positive for every $\tau \in [0,1]$ and $\lambda$,
then the risk of the Bayesian predictive distribution $p_\pi(y \mid x)$ is smaller than
that of $p_{\pi'}(y \mid x)$ for every $\lambda$.
Intuitively speaking, if the estimators $\hat{\lambda}^{\pi}_i (\cdot,\tau)$ based on $\pi$
is superior in the risk \eqref{hyouka-3} for all $\tau \in [0,1]$, then
the Bayesian predictive density $p_\pi(y \mid x)$ is superior in the Kullback--Leibler risk.

\section{Bayesian prediction and estimation}
We introduce a function to represent Bayesian predictive densities and estimators based on $\pi_{\alpha,\beta,\gamma}$.

\vspace{0.5cm}

\begin{definition}
Suppose that $\gamma \in \mathbb{R}^d$, $\gamma_i > 0 ~(i=1,\ldots,d)$,
$x \in \mathbb{R}^d$, $x_{\Cdot} > 0$, and $0< \alpha < x_{\Cdot}$.
Define
\[
K(\gamma, x, \alpha) := \int^\infty_0 u^{\alpha-1} \prod^d_{i=1} \frac{1}{(u/\gamma_i + 1)^{x_i}} \dd u.
\]
\qed
\end{definition}

\vspace{0.5cm}

When $\gamma_1 = \dotsb = \gamma_d$,
\begin{align*}
K(\gamma, x, \alpha) & = \int^\infty_0 \frac{u^{\alpha-1}}{(u/\gamma_1 + 1)^{{x_{\Cdot}}}} \dd u
= \gamma_1^\alpha B (x_{\Cdot} - \alpha, \alpha).
\end{align*}
Thus, $K(\gamma, x, \alpha)$ is a generalization of the beta function.

Lemma \ref{lemma-p_abc} below gives explicit forms of Bayesian predictive densities based on $\pi_\beta$
and $\pi_{\alpha,\beta,\gamma}$.

\vspace{0.5cm}

\begin{lemma}
\label{lemma-p_abc}
Suppose that $z_i(\tau)$ $(i=1,\ldots,d)$ are independent time-inhomogeneous Poisson processes
with mean $t_i(\tau) \lambda_i$.
Let $z_\Delta(\tau) = z(\tau+\Delta) - z(\tau)$,
where $\tau \in [0,1)$ and $\Delta \in (0,1-\tau]$.
\begin{description}
\item{~1)~} The Bayesian predictive density based on the prior
$\pi_\beta(\lambda) = \lambda_1^{\beta_1-1} \dotsb \lambda_d^{\beta_d-1}$,
where $\beta_i > 0$ $(i=1,\ldots,d)$, is given by
\begin{align*}
p_\beta & (z_\Delta(\tau) \mid z(\tau)) =
\prod^d_{i=1}
\left\{
\frac{\Gamma (z_i + (z_\Delta)_i + \beta_i)}
{\Gamma (z_i + \beta_i) (z_\Delta)_i!}
\frac{\{t_i(\tau)\}^{z_i + \beta_i}
\{t_i(\tau+\Delta) - t_i(\tau)\}^{(z_\Delta)_i}}
{\{t_i(\tau+\Delta)\}^{z_i + (z_\Delta)_i + \beta_i}}
\right\},
\end{align*}
which is a product of negative binomial densities.
In particular, when $\tau = 0$ and $\Delta = 1$,
\begin{align*}
p_\beta(y \mid x)
= &
\prod_{i=1}^d \left\{ \frac{\Gamma(x_i + y_i + \beta_i)}{\Gamma(x_i + \beta_i) y_i !}
 \frac{r_i^{x_i+\beta_i} s_i^{y_i}}{(r_i + s_i)^{x_i + y_i + \beta_i}}
\right\},
\end{align*}
where $r_i = t_i(0)$, $r_i+s_i = t_i(1)$, $x=z(1)$, and $y=z_{\Delta=1}(0)$.
\item{~2)~} The Bayesian predictive density based on the prior
$\pi_{\alpha,\beta,\gamma}(\lambda)
= \lambda_1^{\beta_1-1} \dotsb \lambda_d^{\beta_d-1}/
(\lambda_1/\gamma_1 + \cdots + \lambda_d/\gamma_d)^\alpha$,
where $0 < \alpha < \beta_{\Cdot}$, $\beta_i > 0$, and $\gamma_i > 0$ $(i=1,\ldots,d)$,
is given by
\begin{align*}
p_{\alpha,\beta,\gamma} (z_\Delta(\tau) \mid z(\tau)) 
=&\
p_\beta (z_\Delta(\tau) \mid z(\tau))
\dfrac{\displaystyle
\int_0^\infty u^{\alpha-1}
\prod^d_{j=1} \frac{1}{\{\frac{u}{t_j(\tau+\Delta) \gamma_j} + 1\}^{z_j + (z_\Delta)_j + \beta_j}} \dd u
}
{\displaystyle
\int_0^\infty u^{\alpha-1}
\prod^d_{j=1} \frac{1}{\{\frac{u}{t_j(\tau) \gamma_j} + 1\}^{z_j + \beta_j}} \dd u
} \\
=&\
p_\beta (z_\Delta(\tau) \mid z(\tau))
\dfrac{K(t(\tau+\Delta) \gamma, z+z_\Delta+\beta, \alpha)}{K(t(\tau) \gamma, z+\beta, \alpha)},
\end{align*}
where $t \gamma := (t_1 \gamma_1, t_2 \gamma_2, \ldots, t_d \gamma_d)$.

In particular, when $\tau = 0$ and $\Delta = 1$,
\begin{align*}
p_{\alpha,\beta,\gamma} (y \mid x)
=&\ p_\beta(y \mid x)
\dfrac{\displaystyle
\int u^{\alpha-1}
\prod^d_{j=1} \frac{1}{\{\frac{u}{(r_j+s_j)\gamma_j} + 1 \}^{x_j + y_j + \beta_j}} \dd u
}
{\displaystyle
\int u^{\alpha-1}
\prod^d_{j=1} \frac{1}{(\frac{u}{r_j \gamma_j} + 1)^{x_j + \beta_j}} \dd u
} \\
=&\
p_\beta (y \mid x)
\dfrac{K((r+s) \gamma, x+y+\beta, \alpha)}{K(r \gamma, x+\beta, \alpha)},
\end{align*}
where $r_i = t_i(0)$, $r_i+s_i = t_i(1)$, $x=z(0)$, $y = z_{\Delta=1}(0)$,
$r \gamma := (r_1 \gamma_1, \ldots, r_d \gamma_d)$, and
$(r+s) \gamma := ((r_1+s_1) \gamma_1, \ldots, (r_d+s_d) \gamma_d)$.
\end{description}
\qed
\end{lemma}

\vspace{0.5cm}

Lemma \ref{lemma-estimators} below gives explicit forms of Bayesian estimators based on $\pi_\beta$
and $\pi_{\alpha,\beta,\gamma}$.

\vspace{0.5cm}

\begin{lemma}
\label{lemma-estimators}
Suppose that $z_i(\tau)$ $(i=1,\ldots,d)$ are independently distributed according to
the Poisson distribution with mean $t_i(\tau) \lambda_i$.
\begin{description}
 \item{~1)~\;}
The posterior mean
of $\lambda$ with respect to the observation
$z(\tau) = (z_1,\ldots,z_d)$ and
the prior $\pi_\beta(\lambda) = \lambda_1^{\beta_1-1} \dotsb \lambda_d^{\beta_d-1}$,
where $\beta_i > 0$ $(i=1,\ldots,d)$, is given by
\begin{align*}
\hat{\lambda}^{(\beta)}_i(z,\tau) :=& \frac{z_i + \beta_i}{t_i(\tau)}.
\end{align*}
 \item{~2)~\;}
The posterior mean 
of $\lambda$ with respect to the observation
$z(\tau) = (z_1,\ldots,z_d)$ and
the prior $\pi_{\alpha,\beta,\gamma} = \lambda_1^{\beta_1-1} \dotsb \lambda_d^{\beta_d-1}/
(\lambda_1/\gamma_1 + \cdots + \lambda_d/\gamma_d)^\alpha$,
where $0 < \alpha < \beta_{\Cdot}$, $\beta_i > 0$, and $\gamma_i > 0$ $(i=1,\ldots,d)$, is given by
\begin{align*}
\hat{\lambda}_i^{(\alpha,\beta,\gamma)}(z,\tau) :=& \hat{\lambda}^{(\beta)}_i(z,\tau)
\dfrac{\displaystyle
 \int u^{\alpha-1}
\prod^d_{j=1} \frac{1}{\left\{\frac{u}{t_j(\tau) \gamma_j} + 1 \right\}^{z_j + \beta_j +\delta_{ij}}} \dd u}
{\displaystyle
 \int u^{\alpha-1} \prod^d_{j=1} \frac{1}{\left\{\frac{u}{t_j(\tau) \gamma_j} + 1 \right\}^{z_j + \beta_j}} \dd u} \\
=& \hat{\lambda}^{(\beta)}_i(z,\tau) \frac{K(t\gamma,z+\beta+\delta_i, \alpha)}{K(t\gamma,z+\beta,\alpha)},
\end{align*}
where $\delta_{ij}$ is defined to be 1 if $i=j$ and 0 if $i \neq j$,
and $\delta_i$ is defined to be the $d$-dimensional vector whose $i$-th element is $1$ and all other elements are $0$.
\end{description}
\qed
\end{lemma}

\vspace{0.5cm}

Let
\begin{align}
\label{coeff}
f_i(t\gamma, z+\beta, \alpha)
:= \frac{K(t\gamma, z+\beta+\delta_i, \alpha)}{K(t\gamma, z+\beta, \alpha)}.
\end{align}
Then,
\begin{align*}
\hat{\lambda}_i^{(\alpha,\beta,\gamma)}(z,\tau) =& \hat{\lambda}_i^{(\beta)}(z,\tau) f_i(t(\tau) \gamma, z+\beta, \alpha).
\end{align*}
Obviously, $0< f_i(t \gamma, z + \beta,\alpha) < 1$.
This inequality is natural because $\pi_{\alpha, \beta, \gamma}$ is a shrinkage prior.

\vspace{0.5cm}

In particular, if $t_1\gamma_1 = \cdots = t_d\gamma_d$, then
\begin{align*}
f_i(t \gamma, z+\beta, \alpha)
=&\ \frac{(t_1 \gamma_1)^\alpha B(z_{\Cdot} + \beta_{\Cdot} + 1 - \alpha, \alpha)}
{(t_1 \gamma_1)^\alpha B(z_{\Cdot} + \beta_{\Cdot} - \alpha, \alpha)}
= \frac{z_{\Cdot} + \beta_{\Cdot} - \alpha}{z_{\Cdot} + \beta_{\Cdot}},
\end{align*}
which does not depend on $t_1 \gamma_1$.

Now, we give the main theorem.

\vspace{0.5cm}

\begin{thm}
\label{maintheorem}
Suppose that $x_i$ and $y_i$ $(i=1,\ldots,d)$ are independently distributed
according to the Poisson distributions
with mean $r_i \lambda_i$ and $s_i \lambda_i$, respectively.
Let $p_\beta (y \mid x)$ be the Bayesian predictive density based on
$\pi_\beta(\lambda) = \lambda_1^{\beta_1-1} \dotsb \lambda_d^{\beta_d-1}$.
Assume that $\beta_{\Cdot} > 1$.
Let $\pi^*_\beta(\lambda) := \pi_{\alpha,\beta,\gamma}(\lambda) = \lambda_1^{\beta_1-1} \dotsb \lambda_d^{\beta_d-1}/
(\lambda_1/\gamma_1 + \cdots + \lambda_d/\gamma_d)^\alpha$ with
\begin{align*}
\alpha = \beta_{\Cdot} - 1 ~~ \text{and} ~~~
\gamma_i = \frac{1}{r_i} - \frac{1}{r_i+s_i} ~~~~ (i=1,\ldots,d).
\end{align*}
Then, the risk of the Bayesian predictive density
\begin{align*}
p_\beta^* (y \mid x)
= p_\beta (y \mid x)
\dfrac{
K \biggl( \dfrac{s}{r}, x+y+\beta, \alpha \biggr)
}
{
K \biggl( \dfrac{s}{r+s}, x+\beta, \alpha \biggr)
}
\end{align*}
based on $\pi_\beta^*$,
where
\begin{align*}
\frac{s}{r} := \left( \frac{s_1}{r_1}, \ldots, \frac{s_d}{r_d} \right) ~~ \text{and} ~~~
\frac{s}{r+s} := \left( \frac{s_1}{r_1 + s_1}, \ldots, \frac{s_d}{r_d + s_d} \right),
\end{align*}
is smaller than that of $p_\beta (y \mid x)$ for every $\lambda$.
\qed
\end{thm}

\vspace{0.5cm}

If $d \geq 3$, there exists a Bayesian predictive density dominating that
based on the Jeffreys prior \eqref{jeffreysx}
for $p(x \mid \lambda)$ because $\beta_{\Cdot} = d/2 > 1$,
as in the simple setting with $r_1 = \cdots = r_d$ and $s_1 = \cdots = s_d$ studied in \cite{Komaki:TheAnnalsOfStatistics:2004}.
Note that the prior $\pi^*_\beta$ depends on $r$ and $s$.

Before proving Theorem 1, we prepare Lemmas \ref{lemma-koukan-0} and \ref{lemma-Ksekibun} below.

\vspace{0.5cm}

\begin{lemma}
\label{lemma-koukan-0}
Let $h(x)$ be a real valued function of $x = (x_1,\ldots,x_d) \in \mathbb{N}_0^d$,
where $\mathbb{N}_0$ is the set of nonnegative integers.
Suppose that $x_i$ $(i=1,\ldots, d)$ are independently distributed according to
the Poisson distribution with mean $\lambda_i$.
If $\E \bigl[|x_i h(x)| \mid \lambda \bigr] < \infty$, then
\begin{align*}
\E[x_i h(x) \mid \lambda] =& \E[\lambda_i h(x + \delta_i) \mid \lambda].
\end{align*}
\qed
\end{lemma}

\vspace{0.5cm}

\begin{lemma}
\label{lemma-Ksekibun}
Suppose that $\gamma \in \mathbb{R}^d$, $\gamma_i > 0 ~(i=1,\ldots,d)$,
$x \in \mathbb{R}^d$, $x_{\Cdot} > 0$, and $0< \alpha < x_{\Cdot}$.
Then, the following relations hold.
\begin{description}
\item{~1)~}
\begin{equation}
\alpha K(\gamma, x, \alpha) = \sum_{i=1}^d \frac{x_i}{\gamma_i} K (\gamma, x + \delta_i, \alpha+1).
\label{lemma-1-0}
\end{equation}
\item{~2)~}
\begin{equation}
\gamma_i K(\gamma, x, \alpha) = K(\gamma, x + \delta_i, \alpha+1) + \gamma_i K (\gamma, x + \delta_i, \alpha).
\label{lemma-2-0}
\end{equation}
\item{~3)~}
Let $b = (b_1,b_2,\ldots,b_d) \in \mathbb{R}^d$.
Then,
\begin{align}
\sum_{i=1}^d b_i K (\gamma, x+ \delta_i, \alpha)
= & \sum_{i=1}^d \left(\frac{b_{\Cdot} x_i}{\alpha \gamma_i} - \frac{b_i}{\gamma_i} \right) K(\gamma, x + \delta_i, \alpha + 1) .
\label{lemma-3-0}
\end{align}
\end{description}
\qed
\end{lemma}

\vspace{0.5cm}

\begin{proof}[Proof of Theorem \ref{maintheorem}]

Let
\begin{align*}
\frac{1}{t_i(\tau)} =& \frac{1}{r_i}(1-\tau) + \frac{1}{r_i+s_i} \tau ~~~ \mbox{for~~} \tau \in [0,1].
\end{align*}
Then,
\begin{align*}
t_i(\tau) =& r_i \frac{\displaystyle 1 + \frac{s_i}{r_i}}{\displaystyle 1 + \frac{s_i}{r_i}(1-\tau)}
\end{align*}
is a smooth monotonically increasing function of $\tau \in [0,1]$ satisfying $t_i(0) = r_i$ and $t_i(1) = r_i + s_i$.
Here, $\dot{t}_i/t_i = \gamma_i t_i$
since $\frac{\scriptsize \dd}{\scriptsize \dd \tau} \{1/t_i(\tau)\} = -\dot{t_i}/t_i^2 = -1/r_i + 1/(r_i + s_i) = -\gamma_i$.
We call $\tau$ the harmonic time because $\tau$ is the weight of the weighted harmonic mean $t_i(\tau)$ of $r_i$ and $r_i+s_i$.

By Lemma \ref{lemma-estimators},
the posterior mean of $\lambda$ with respect to $\pi_\beta$ is
\begin{align*}
\hat{\lambda}^{(\beta)}_i(z,\tau) =&\ \frac{z_i + \beta_i}{t_i(\tau)},
\end{align*}
and the posterior mean $\lambda$ with respect to $\pi_\beta^*$ is 
\begin{align*}
\hat{\lambda}_i^{(\beta*)}(z,\tau) =&\ \hat{\lambda}_i^{(\beta)}(z,\tau) f_i(\gamma t(\tau),z+\beta,\beta_{\Cdot}-1)
= \frac{z_i + \beta_i}{t_i(\tau)} f_i(\gamma t(\tau),z+\beta,\beta_{\Cdot}-1).
\end{align*}
Thus, from Lemma \ref{smallprediction}, it is sufficient to show that
\begin{align}
\sum_i \E \biggl[ &\ \dot{t}_i(\tau) \left\{ \hat{\lambda}_i^{(\beta)}(z(\tau),\tau)
- \hat{\lambda}_i^{(\beta*)}(z(\tau),\tau)
- \lambda_i \log \frac{\hat{\lambda}_i^{(\beta)}(z(\tau),\tau)}{\hat{\lambda}_i^{(\beta*)}(z(\tau),\tau)} \right\}
\, \bigg| \, \lambda \biggr] \notag \\
=& \sum_i \E \biggl[ \dot{t}_i(\tau) \frac{z_i(\tau) + \beta_i}{t_i(\tau)}
\Bigl\{ 1 - f_i(\gamma t(\tau), z(\tau)+\beta, \beta_{\Cdot}-1) \Bigr\} \notag \\
& ~~~~~~~~ + \frac{\dot{t}_i(\tau)}{t_i(\tau)} t_i(\tau) \lambda_i
\log f_i(\gamma t(\tau), z(\tau)+\beta, \beta_{\Cdot}-1) \,\bigg|\, \lambda \biggr]
\label{47-0}
\end{align}
is positive for every $\tau \in [0,1]$ and $\lambda$.
Define $\bar{f}_i(\gamma t, z-\delta_i+\beta, \alpha) = f_i(\gamma t, z-\delta_i+\beta, \alpha)$
if $z_i \geq 1$ and $\bar{f}_i(\gamma t, z-\delta_i+\beta, \alpha) = 1$
if $z_i = 0$.
Then, by Lemma \ref{lemma-koukan-0}, \eqref{47-0} is equal to
\begin{align}
\E \biggl[ & \sum_i \frac{\dot{t}_i(\tau)}{t_i(\tau)} (z_i(\tau) + \beta_i)
\Bigl\{ 1 - f_i(\gamma t(\tau), z(\tau)+\beta, \beta_{\Cdot}-1) \Bigr\} \notag \\
&
+ \sum_i \frac{\dot{t}_i(\tau)}{t_i(\tau)} z_i(\tau)
\log \bar{f}_i(\gamma t(\tau), z(\tau)-\delta_i+\beta, \beta_{\Cdot}-1) \, \bigg| \, \lambda \biggr]
\label{54-0}
\end{align}
since $z_i(\tau)$ is independently distributed according to the Poisson distribution with mean $t_i(\tau) \lambda_i$.
Note that \eqref{54-0} is the expectation of functions of $z(\tau)$ not depending on $\lambda$.

First, we evaluate the first term in the expectation in \eqref{54-0}.
By using \eqref{coeff} and \eqref{lemma-2-0},
\begin{align}
1- f_i&(\gamma t, z+\beta,\beta_{\Cdot}-1)
= 1 - \frac{K(\gamma t, z + \beta + \delta_i, \beta_{\Cdot}-1)}{K(\gamma t, z + \beta, \beta_{\Cdot}-1)} \notag \\
=&
1 - \frac{K(\gamma t, z + \beta, \beta_{\Cdot}-1) - \dfrac{1}{\gamma_i t_i} K(\gamma t, z + \beta + \delta_i, \beta_{\Cdot})}
{K(\gamma t, z + \beta, \beta_{\Cdot}-1)} \notag \\
=& \frac{K(\gamma t, z + \beta + \delta_i, \beta_{\Cdot})}{\gamma_i t_i K(\gamma t, z + \beta, \beta_{\Cdot}-1)}.
\label{fK-2}
\end{align}
From $\dot{t}_i/t_i = \gamma_i t_i$ and \eqref{fK-2}, we have
\begin{align}
\label{keisu}
\sum_i & \frac{\dot{t}_i}{t_i} (z_i + \beta_i) \{1-f_i(\gamma t, z+\beta, \beta_{\Cdot}-1)\}
= \frac{\sum_i (z_i + \beta_i) K(\gamma t, z+\beta+\delta_i, \beta_{\Cdot})}
{K(\gamma t, z+\beta, \beta_{\Cdot}-1)}.
\end{align}
If $z_{\Cdot} = 0$, then $z_1 = \cdots = z_d = 0$ and
\begin{align*}
\sum_i & \frac{\dot{t}_i}{t_i} (z_i + \beta_i) \{1-f_i(\gamma t, z+\beta, \beta_{\Cdot}-1)\}
= \frac{\sum_i \beta_i K(\gamma t, \beta+\delta_i, \beta_{\Cdot})}
{K(\gamma t, \beta, \beta_{\Cdot}-1)} > 0.
\end{align*}
If $z_{\Cdot} \geq 1$,
from \eqref{keisu}, \eqref{lemma-3-0}, and \eqref{lemma-1-0}, we have
\begin{align*}
\sum_i \frac{\dot{t}_i}{t_i} & (z_i + \beta_i) \{1-f_i(\gamma t, z+\beta, \beta_{\Cdot}-1)\} \\
=&
\frac{\sum_i \left\{ \dfrac{(z_{\Cdot} + \beta_{\Cdot})(z_i + \beta_i)}{\beta_{\Cdot} \gamma_i t_i}
 - \dfrac{z_i + \beta_i}{\gamma_i t_i} \right\}
K(\gamma t, z+\beta+\delta_i, \beta_{\Cdot}+1)}
{K(\gamma t, z+\beta, \beta_{\Cdot}-1)} \\
=& \frac{\dfrac{z_{\Cdot}}{\beta_{\Cdot}} \sum_i \dfrac{z_i + \beta_i}{\gamma_i t_i}
K(\gamma t, z+\beta+\delta_i, \beta_{\Cdot}+1)}
{K(\gamma t, z+\beta, \beta_{\Cdot}-1)} \\
=&
\dfrac{z_{\Cdot}}{\beta_{\Cdot}} 
\frac{\beta_{\Cdot} K(\gamma t, z+\beta, \beta_{\Cdot})}
{K(\gamma t, z+\beta, \beta_{\Cdot}-1)}
= z_{\Cdot}
\frac{K(\gamma t, z+\beta, \beta_{\Cdot})}
{K(\gamma t, z+\beta, \beta_{\Cdot}-1)}.
\end{align*}

Next, we evaluate the second term in the expectation in \eqref{54-0}.
We have
\begin{align*}
\frac{\dot{t}_i}{t} & z_i \log \bar{f}_i(\gamma t, z+\beta-\delta_i, \beta_{\Cdot}-1)
= - \gamma_i t_i z_i \log \left\{\frac{1}{\bar{f}_i(\gamma t, z + \beta -\delta_i, \beta_{\Cdot}-1)} - 1 + 1 \right\}.
\end{align*}
From \eqref{coeff} and \eqref{fK-2}, if $z_i \geq 1$,
\begin{align*}
\frac{1}{\bar{f}_i(\gamma t, z+\beta-\delta_i, \beta_{\Cdot}-1)} -1 
=
\frac{K(\gamma t, z+\beta, \beta_{\Cdot})}{\gamma_i t_i K(\gamma t, z+\beta, \beta_{\Cdot}-1)}.
\end{align*}
Thus, when $z_i \geq 1$,
\begin{align*}
\frac{\dot{t}_i}{t} z_i \log \bar{f}_i(\gamma t, z+\beta-\delta_i, \beta_{\Cdot}-1)
=& - \gamma_i t_i z_i
\log \left\{\frac{K(\gamma t, z +\beta, \beta_{\Cdot})}{\gamma_i t_i K(\gamma t, z +\beta, \beta_{\Cdot}-1)} +1 \right\} \\
>&
- z_i \frac{K(\gamma t, z + \beta, \beta_{\Cdot})}{K(\gamma t, z+\beta, \beta_{\Cdot}-1)}.
\end{align*}
When $z_i = 0$, the equality
\begin{align*}
\frac{\dot{t}_i}{t} & z_i \log \bar{f}_i(\gamma t, z+\beta-\delta_i, \beta_{\Cdot}-1)
= - z_i \frac{K(\gamma t, z + \beta, \beta_{\Cdot})}{K(\gamma t, z+\beta, \beta_{\Cdot}-1)} = 0
\end{align*}
obviously holds.
Thus, for every $z$,
\begin{align*}
\sum_i \frac{\dot{t}_i}{t} & z_i \log \bar{f}_i(\gamma t, z+\beta-\delta_i, \beta_{\Cdot}-1)
\geq
- z_{\Cdot} \frac{K(\gamma t, z + \beta, \beta_{\Cdot})}{K(\gamma t, z+\beta, \beta_{\Cdot}-1)}.
\end{align*}
The inequality is strict if $z_{\Cdot} \geq 1$.

Hence, for every $z \in \mathbb{N}_0^d$,
\begin{align*}
\sum_i & \frac{\dot{t}_i}{t_i} (z_i + \beta_i) \{1-f_i(\gamma t, z+\beta, \beta_{\Cdot}-1)\}
+ \sum_i \frac{\dot{t}_i}{t} z_i \log \bar{f}_i(\gamma t, z+\beta-\delta_i, \beta_{\Cdot}-1)
> 0
\end{align*}

Therefore, \eqref{54-0} is greater than $0$ for every $\tau \in [0,1]$ and $\lambda$. 
Thus, we have proved the desired result.
\end{proof}

\section{Relative invariance of the prior along with the harmonic time $\tau$}

In this section, $\pi^*_{\beta}$ in Theorem 1 is denoted by $\pi^*_{\beta,r,s}$
to indicate its dependence on $r = (r_1,\ldots,r_d)$ and $s = (s_1,\ldots,s_d)$ explicitly.
The prior $\pi^*_{\beta,r,s}$ depends on $r$ and $s$ through
$(1/r_1 - 1/(r_1+s_1), \ldots, 1/r_d - 1/(r_d+s_d))$
because
$\pi^*_{\beta,r,s} = \pi_{\alpha,\beta,\gamma}$ with $\alpha = \beta_{\Cdot}$ and $\gamma_i = 1/r_i - 1/(r_i+s_i)$.
If there exists a constant $c>0$ such that
\[ 
\frac{1}{r_i'} - \frac{1}{r_i'+s_i'}
= c \left(\frac{1}{r_i} - \frac{1}{r_i+s_i}\right)
\]
for $i=1,\ldots,d$,
then $\pi^*_{\beta,r,s}$ is proportional $\pi^*_{\beta,r',s'}$
because $\pi_{\alpha,\beta,c\gamma} \propto \pi_{\alpha,\beta,\gamma}$.

Consider the harmonic time $\tau \in (-\infty, \min_i (r_i/s_i) + 1)$ satisfying
\begin{align*}
\frac{1}{t_i(\tau)}
=& \frac{1}{r_i}(1-\tau) + \frac{1}{r_i+s_i} \tau.
\end{align*}
The discussions in previous sections are essentially valid if the time interval $[0,1]$
is extended to $(-\infty, \min_i (r_i/s_i) + 1)$.
Suppose that we observe $z(a)$, where $a \in (-\infty, \min_i (r_i/s_i) + 1)$,
and predict $z(b) - z(a)$, where $b \in (a, \min_i (r_i/s_i) + 1)$.
Since
\begin{align*}
\frac{1}{t_i(a)} - \frac{1}{t_i(b)}
=& \left\{\frac{1}{r_i}(1-a) + \frac{1}{r_i+s_i}a \right\}
- \left\{\frac{1}{r_i}(1-b) + \frac{1}{r_i+s_i}b \right\} \\
=& (b - a) \left( \frac{1}{r_i} - \frac{1}{r_i+s_i} \right),
\end{align*}
the prior $\pi^*_{\beta,r/(b-a),s/(b-a)}$ for this prediction problem
is proportional to the prior $\pi^*_{\beta,r,s}$ for the original prediction problem
in which we observe $z(0)$ and predict $z(1) - z(0)$.
In this sense, the prior constructed by Theorem 1 is relatively invariant along with the harmonic time $\tau$.
This relative invariance corresponds to the fact that
the estimators $\hat{\lambda}^{(\beta *)}_i (\cdot,\tau)$ based on $\pi^*_{\beta,r,s}$
is superior in the risk \eqref{hyouka-3} for all $\tau$
and is one reason why the harmonic time $\tau$
is useful to investigate the original prediction problem.

Next, we discuss the relation between the results in previous sections and the asymptotic theory \citep{KBA2014}
for general models when $x(i)$ $(i=1,\ldots,N)$ and $y$ have different distributions
$p(x \mid \theta)$ and $p(y \mid \theta)$ with the same parameter $\theta$.
The predictive metric $\gxy{ij}{}$ is defined by
$\sum_{k,l} g_{ik} \tilde{g}^{kl} g_{jl}$, where $(g_{ij})$ and $(\tilde{g}_{ij})$ are the Fisher information matrices
for $p(x \mid \theta)$ and $p(y \mid \theta)$, respectively, and
the $d \times d$ matrix $(\gy{}{ij})$ is the inverse matrix of $(\gy{ij}{})$.
In the asymptotic theory,
the predictive metric $\gxy{ij}{}$
and the volume element $|\gxy{}{}|^{1/2} \dd \theta^1 \cdots \dd \theta^d$
of it correspond to the Fisher--Rao metric and the Jeffreys prior, respectively, in the conventional setting.

In the prediction problem for independent time-inhomogeneous Poisson processes with the harmonic time $\tau$,
the Fisher information matrix $(\gx{ij}{})$ for $p(z(\tau) \mid \lambda)$
and the Fisher information matrix $(\gy{ij}{})$ for $p(z_\Delta(\tau) \mid \lambda)$
are given by
\begin{align*}
\gx{ij}{}(\lambda;\tau) =
\left\{
\begin{array}{cc}
\displaystyle \frac{t_i(\tau)}{\lambda_i} & (i=j) \\[0.5cm]
0 & (i \neq j)
\end{array}
\right. 
\end{align*}
and
\begin{align*}
\gy{ij}{}(\lambda;\tau) =
\left\{
\begin{array}{cc}
\displaystyle \frac{t_i(\tau+\Delta) - t_i(\tau)}{\lambda_i} & (i=j) \\[0.5cm]
0 & (i \neq j)
\end{array}
\right., 
\end{align*}
respectively.
When $\Delta$ is small,
$\gy{ii}{}(\lambda;\tau) = \dot{t}_i(\tau) \Delta / \lambda_i + \mathrm{o}(\Delta)$.
We define the infinitesimal predictive metric by
\begin{align}
\label{poissonpm}
\gxy{ij}{}(\lambda;\tau) := \lim_{\Delta \rightarrow 0} \Delta \sum_{k,l} \gx{ik}{}\gy{}{ij}\gx{jl}{}
= \left\{
\begin{array}{cc}
\displaystyle \frac{\{t_i(\tau)\}^2}{\dot{t}_i(\tau) \lambda_i}
= \frac{r_i(r_i+s_i)}{\lambda_i} & (i=j) \\[0.5cm]
0 & (i \neq j)
\end{array}
\right.,
\end{align}
which is the limit of the predictive metric as $\Delta \rightarrow 0$.
The last equality in \eqref{poissonpm} is because the relations
$\dot{t_i}^2(\tau)/t_i(\tau) = r_i(r_i+s_i)$ $(i=1,\ldots,d)$
holds for the harmonic time $\tau$.
The volume element prior based on $\gxy{ij}{}(\lambda;\tau)$ is defined by
$\pi_\text{P}(\lambda;\tau) = |\gxy{ij}{}(\lambda;\tau)|^{1/2}$
and is proportional to the Jeffreys prior $\pi_\text{J}(\lambda) \propto \prod_i {\lambda_i}^{-1/2}$.
Thus, when the harmonic time $\tau$ is adopted, the infinitesimal predictive metric
and the volume element prior based on it do not depend on $\tau$.
Intuitively speaking, the geometrical structures of infinitesimal prediction are identical for all $\tau$.
Hence, there exists a prior superior for infinitesimal predictions for all $\tau$
and the prior is also superior for the original prediction problem.
More specifically, the ratio $\pi^*_{\beta,r,s}(\lambda)/\pi_\text{P}(\lambda;\tau)$ does not depend on $\tau$
and is a nonconstant positive superharmonic function with respect to the predictive metric $\gxy{ij}{}(\lambda;\tau)$
for every $\tau$, see \cite{KBA2014} for details.
This property of the harmonic time $\tau$ is closely related to
the relative invariance of the prior $\pi^*_{\beta,r,s}$ along with $\tau$.

\newpage
\appendix
\def\thesection{Appendix.}

\section{Proofs of Lemmas}
\begin{proof}[Proof of Lemma \ref{smallprediction}]

\noindent
1)~
First, we prove \eqref{hyouka-2}.
We have
\begin{align*}
\E\Bigl[ & D\{p(z_\Delta(\tau) \mid \lambda), p_\pi(z_\Delta(\tau) \mid z(\tau))\} \; \Big| \; \lambda \Bigr] =
\sum_{z(\tau), z_\Delta(\tau)} p (z(\tau), z_\Delta(\tau) \mid \lambda)
\log \frac{p(z_\Delta(\tau) \mid \lambda)}{p_{\pi} (z_\Delta(\tau) \mid z(\tau))} \notag \\
=& \sum_{z(\tau), z_\Delta(\tau)} p(z(\tau), z_\Delta(\tau) \mid \lambda) \log p(z(\tau), z_\Delta(\tau) \mid \lambda)
 - \sum_{z(\tau)} p(z(\tau) \mid \lambda) \log p (z(\tau) \mid \lambda) \notag \\
& - \sum_{z(\tau), z_\Delta(\tau)} p (z(\tau), z_\Delta(\tau) \mid \lambda) \log p_{\pi} (z(\tau), z_\Delta(\tau)) 
+ \sum_{z(\tau)} p(z(\tau) \mid \lambda) \log p_{\pi} (z(\tau)).
\end{align*}
The conditional density $p(z(\tau) \mid z(\tau+\Delta), \lambda)$ does not depend on $\lambda$
because of the sufficiency of $z(\tau+\Delta) = z(\tau) + z_\Delta(\tau)$.
Thus,
\begin{align}
\E\Bigl[ & D\{p(z_\Delta(\tau) \mid \lambda), p_\pi(z_\Delta(\tau) \mid z(\tau))\} \; \Big| \; \lambda \Bigr] \notag \\
= & \sum_{z(\tau), z(\tau+\Delta)} p(z(\tau), z(\tau+\Delta) \mid \lambda)
\log \{p (z(\tau+\Delta) \mid \lambda) p(z(\tau) \mid z(\tau+\Delta)) \} \notag \\
& - \sum_{z(\tau)} p(z(\tau) \mid \lambda) \log p (z(\tau) \mid \lambda) \notag \\
& - \sum_{z(\tau), z(\tau+\Delta)} p (z(\tau), z(\tau+\Delta) \mid \lambda) \log \{p_{\pi} (z(\tau+\Delta))
p(z(\tau) \mid z(\tau+\Delta)) \} \notag \\
& + \sum_{z(\tau)} p(z(\tau) \mid \lambda) \log p_{\pi} (z(\tau)) \notag \\
= & \sum_{z(\tau+\Delta)} p(z(\tau+\Delta) \mid \lambda) \log p (z(\tau+\Delta) \mid \lambda)
 - \sum_{z(\tau)} p(z(\tau) \mid \lambda) \log p (z(\tau) \mid \lambda) \notag \\
& - \sum_{z(\tau+\Delta)} p (z(\tau+\Delta) \mid \lambda) \log p_{\pi} (z(\tau+\Delta))
+ \sum_{z(\tau)} p(z(\tau) \mid \lambda) \log p_{\pi} (z(\tau)).
\label{eq:diff}
\end{align}
Therefore, we have
\begin{align}
\frac{\partial}{\partial \Delta} & 
\E\Bigl[ D\{p(z_\Delta(\tau) \mid \lambda), p_\pi(z_\Delta(\tau) \mid z(\tau))\} \; \Big| \; \lambda \Bigr]
\bigg|_{\Delta = 0} \notag \\
= &
\frac{\partial}{\partial \tau} \sum_z p (z(\tau)\mid\lambda) \log \frac{p(z(\tau) \mid \lambda)}{p_\pi (z(\tau))}
= \frac{\partial}{\partial \tau} D \{ p(z(\tau) \mid \lambda), p_\pi(z(\tau)) \}
\label{18-1}
\end{align}
because
$\E\Bigl[ D\{p(z_\Delta(\tau) \mid \lambda), p_\pi(z_\Delta(\tau) \mid z(\tau))\} \; \Big| \; \lambda \Bigr] = 0$
when $\Delta = 0$.

Next, we prove \eqref{hyouka-3}.
We have
\begin{align}
\frac{\partial }{\partial  \tau} & p (z(\tau) \mid \lambda)
= \frac{\dd}{\dd \tau} \prod^d_{i=1} \frac{ \{ t_i(\tau) \lambda_i \}^{z_i}} {z_i!} \ee^{ - t_i (\tau) \lambda_i}
\notag \\
= & \sum^d_{j=1} \left[ \prod^d_{i=1} z_j \frac{ \{ t_i(\tau) \lambda_i \}^{z_i - \delta_{ij}} }{z_i!}
  \dot{t}_j (\tau) \lambda_j \ee^{ -t_i (\tau) \lambda_i}
- \prod^d_{i=1} \frac{ \{ t_i(\tau) \lambda_i \} ^{z_i} } {z_i!} 
  \dot{t}_j (\tau) \lambda_j \ee^{ -t_i (\tau) \lambda_i} \right]
\notag  \\
= & \sum^d_{j=1} \left[ \prod^d_{i=1} z_j \frac{\dot{t}_j(\tau)}{t_j(\tau)} \frac{ \{ t_i(\tau) \lambda_i \}^{z_i}}{z_i!}
     \ee^{ -t_i (\tau) \lambda_i }
- \prod^d_{i=1} \frac{\dot{t}_j(\tau)}{t_j(\tau)} t_j(\tau) \lambda_j \frac{ \{ t_i ( \tau) \lambda_i \}^{z_i}}{z_i!}
\ee^{-t_i(\tau) \lambda_i}\right]
\notag  \\
= & \sum^d_{j=1} \frac{\dot{t}_j(\tau)}{t_j(\tau)} \{z_j - t_j(\tau) \lambda_j\} p(z(\tau) \mid \lambda).
\label{eq-bibun}
\end{align}
Similarly,
\begin{align}
\frac{\partial }{\partial \tau} p_\pi (z (\tau))
= & \sum^d_{j=1} \frac{\dot{t}_j(\tau)}{t_j(\tau)} \{z_j - t_j(\tau) \hat{\lambda}^\pi_j(z,\tau)\}
p_\pi (z (\tau)).
\label{eq-bibundouyou}
\end{align}
From Lemma \ref{lemma-koukan-0},
\begin{align}
\sum_z & \sum^{d}_{j=1} \{z_j - t_j (\tau) \lambda_j \} p (z(\tau) \mid \lambda) \log p_{\pi} (z(\tau)) \notag \\
= & \sum_z \sum^{d}_{j=1} t_j (\tau) \lambda_j p ( z (\tau) \mid \lambda )
\log \frac { p_\pi ( z (\tau) + \delta_j ) }{ p_\pi ( z (\tau) ) }.
\label{13-A}
\end{align}
Since
\begin{align*}
p_\pi (z (\tau) + \delta_j)
=& \int \prod^{d}_{i=1} \frac{ \{ t_i (\tau) \lambda_i \}^{z_i + \delta_{ij}} }{(z_i + \delta_{ij}) !}
    \ee^{ - t_i \lambda_i } \pi ( \lambda) \dd \lambda \\
=& \int \frac{ t_j ( \tau) \lambda_j}{z_j + 1} \prod^{\dd}_{i=1} 
    \frac{ \{t_i(\tau) \lambda_i \}^{z_i} }{z_i!} \ee^{-t_i(\tau) \lambda_i} \pi (\lambda) \dd \lambda,
\end{align*}
we have
\begin{align}
\frac{p_\pi (z(\tau) + \delta_j )}{ p_\pi (z(\tau))} = \frac{t_j(\tau) \hat{\lambda}^\pi_j(z,\tau)}{z_j + 1}.
\label{13-B}
\end{align}
From \eqref{eq-bibun}, \eqref{eq-bibundouyou}, \eqref{13-A}, \eqref{13-B}, and Lemma \ref{lemma-koukan-0},
\begin{align*}
\frac{\partial}{\partial \tau} & \sum_z p (z(\tau)\mid\lambda) \log p_\pi (z(\tau)) \notag \\
=& \sum_z \left\{\frac{\partial}{\partial \tau} p(z(\tau) \mid \lambda) \right\} \log p_\pi (z(\tau)) +
 \sum_z p(z(\tau)\mid\lambda) \frac{\frac{\partial}{\partial \tau} p_\pi (z(\tau))}{p_\pi (z(\tau))} \notag \\
=& \sum_z \sum^d_{j=1} \frac{\dot{t}_j(\tau)}{t_j(\tau)} 
t_j (\tau) \lambda_j p ( z (\tau) \mid \lambda )
\log \frac { p_\pi ( z (\tau) + \delta_j ) }{ p_\pi ( z (\tau) ) } \notag \\
&+ \sum_z \sum^d_{j=1} p(z(\tau)\mid\lambda) \frac{\dot{t}_j(\tau)}{t_j(\tau)}
\{z_j - t_j(\tau) \hat{\lambda}^\pi_j(z,\tau)\} \notag \\
=& \sum_z \sum^d_{j=1} p ( z (\tau) \mid \lambda ) \dot{t}_j(\tau)
\lambda_j \log \frac{t_j (\tau) \hat{\lambda}^\pi_j (z,\tau)}{z_j + 1}
+ \sum_z \sum^d_{j=1} p(z(\tau)\mid\lambda) \dot{t}_j(\tau) \{\lambda_j - \hat{\lambda}^\pi_j(z,\tau)\}.
\end{align*}
Similarly we have,
\begin{align}
\frac{\partial}{\partial \tau} & \sum_z p (z(\tau)\mid\lambda) \log p (z(\tau) \mid \lambda)
= \sum_z \sum^d_{j=1} p ( z (\tau) \mid \lambda ) \dot{t}_j(\tau)
\lambda_j \log \frac{t_j (\tau) \lambda_j}{z_j + 1}. \notag
\end{align}
Thus,
\begin{align*}
\frac{\partial}{\partial \tau} & \sum_z p (z(\tau)\mid\lambda) \log \frac{p(z(\tau) \mid \lambda)}{p_\pi (z(\tau))} \notag \\
=& \sum_z p ( z (\tau) \mid \lambda )
\sum^d_{j=1} \dot{t}_j(\tau) \lambda_j 
\left\{
\frac{\hat{\lambda}_j^\pi(z,\tau)}{\lambda_j} - 1 -\log \frac{\hat{\lambda}^\pi_j(z,\tau)}{\lambda_j}
\right\}.
\end{align*}

\noindent
2)~
From \eqref{eq:diff}, we have
\begin{align*}
\E \bigl[ & D(p(y \mid \lambda), p_{\pi'} ( y \mid x)) \, \big| \, \lambda \bigr]
- \E \bigl[ D (p(y \mid \lambda), p_{\pi} (y \mid x)) \, \big| \, \lambda \bigr] \notag \\
=& \E \Bigl[ D \bigl\{ p(z_{\Delta=1}(0) \mid \lambda), p_{\pi'} (z_{\Delta=1}(0) \mid z(0)) \bigr\} \, \Big| \, \lambda \Bigr] \notag \\
& - \E \Bigl[ D \bigl\{ p(z_{\Delta=1}(0) \mid \lambda), p_{\pi} (z_{\Delta=1}(0) \mid z(0)) \bigr\} \, \Big| \, \lambda \Bigr] \notag \\
= & \int^1_0 \frac{\partial}{\partial \tau} \sum_{z(\tau)} p(z(\tau) \mid \lambda ) \log p_\pi (z(\tau)) \dd \tau
- \int^1_0 \frac{\partial}{\partial \tau} \sum_{z(\tau)} p(z(\tau) \mid \lambda) \log p_{\pi'} (z(\tau)) \dd \tau.
\end{align*}
Thus, we obtain the desired results \eqref{lemma1-2-1} and \eqref{lemma1-2-2}
from \eqref{hyouka-2} and \eqref{hyouka-3}, respectively.
\end{proof}
\vspace{0.5cm}

\begin{proof}[Proof of Lemma \ref{lemma-p_abc}]  

\noindent
1)~
Let $z_i = z_i(\tau)$ and $z'_i = (z_\Delta)_i(\tau)$.
Then, we have
\begin{align*}
\int p(z &\mid \lambda) \pi_\beta(\lambda) \dd \lambda
= \int \prod^d_{i=1} \frac{\{t_i(\tau)\}^{z_i} \lambda_i^{z_i + \beta_i - 1}}{z_i!}
\ee^{-t_i(\tau) \lambda_i} \dd \lambda_1 \dotsb \dd \lambda_d \\
=& \prod^d_{i=1} \frac{\{t_i(\tau)\}^{z_i} \Gamma(z_i + \beta_i)}
{z_i! t_i(\tau)^{z_i + \beta_i}}
\end{align*}
and
\begin{align*}
\int p(z,&\ z'\mid \lambda) \pi_\beta(\lambda) \dd \lambda \\
=& \int \prod^d_{i=1} \frac{\{t_i(\tau)\}^{z_i}
\{t_i(\tau+\Delta) - t_i(\tau)\}^{z'_i}
\lambda_i^{z_i + z'_i  + \beta_i - 1}}{z_i! z'_i!}
\ee^{- t_i(\tau+\Delta) \lambda_i}
\dd \lambda_1 \dotsb \dd \lambda_d \\
=& \prod^d_{i=1} 
 \frac{\{t_i(\tau)\}^{z_i} \{t_i(\tau+\Delta) - t_i(\tau)\}^{z'_i}
 \Gamma(z_i + z'_i + \beta_i)}
{z_i! z'_i! \{t_i(\tau+\Delta)\}^{z_i + z'_i + \beta_i}}.
\end{align*}
From $p_\beta(z' \mid z) = p_\beta(z, z')/p_\beta(z)$, we have the desired result.

\noindent
2)~ 
If $\gamma_i > 0$ $(i=1,\ldots,d)$ and $\alpha > 0$,
\begin{align*} 
\int_0^\infty u^{\alpha-1} \exp \left(- u \sum_{i=1}^d \frac{\lambda_i}{\gamma_i} \right) \dd u
= \frac{\Gamma (\alpha)}{(\sum_{i=1}^d \frac{\lambda_i}{\gamma_i})^\alpha}.
\end{align*}
Thus,
\begin{align}
 \pi_{\alpha,\beta,\gamma}(\lambda) =& \frac{\prod_{i=1}^d \lambda_i^{\beta_i-1}}{\Gamma (\alpha)}
\int_0^\infty u^{\alpha-1} \exp \left(- u \sum_{j=1}^d \frac{\lambda_j}{\gamma_j} \right) \dd u.
\label{priorintegral} 
\end{align}
Therefore, since
\begin{align*}
\Gamma(\alpha) & p_{\alpha,\beta,\gamma}(z)
= \Gamma(\alpha) \int p(z \mid \lambda) \pi_{\alpha,\beta,\gamma}(\lambda) \dd \lambda \\
=& \int \prod^d_{i=1} \frac{\{t_i(\tau)\}^{z_i} \lambda_i^{z_i +\beta_i-1}}{z_i!} \ee^{-t_i(\tau) \lambda_i}
\int^\infty_0 u^{\alpha-1} \exp \biggl( - u \sum_j \frac{\lambda_j}{\gamma_j} \biggr) \dd u
\dd \lambda_1 \dotsb \dd \lambda_d \\
=& \int_0^\infty u^{\alpha-1} \int \prod^d_{i=1} \frac{\{t_i(\tau)\}^{z_i} \lambda_i^{z_i + \beta_i - 1}}{z_i!}
\ee^{-\{\frac{u}{\gamma_i} + t_i(\tau)\}\lambda_i} \dd \lambda_1 \dotsb \dd \lambda_d \dd u \\
=& \int_0^\infty u^{\alpha-1} \prod^d_{i=1} \frac{\{t_i(\tau)\}^{z_i} \Gamma(z_i + \beta_i)}
{z_i! \{\frac{u}{\gamma_i} + t_i(\tau)\}^{z_i + \beta_i}} \dd u \\
=& \left[ \prod^d_{i=1} \frac{\{t_i(\tau)\}^{z_i} \Gamma (z_i + \beta_i)}{z_i! \{t_i(\tau)\}^{z_j + \beta_i}} \right]
\int_0^\infty u^{\alpha-1}
\prod^d_{j=1} \frac{1}{\{\frac{u}{t_i(\tau) \gamma_i} + 1\}^{z_j + \beta_i}} \dd u
\end{align*}
and
\begin{align*}
\Gamma(\alpha) & p_{\alpha,\beta,\gamma}(z,z')
= \Gamma(\alpha) \int p(z,z' \mid \lambda) \pi_{\alpha,\beta,\gamma}(\lambda) \dd \lambda \\
=& \int \prod^d_{i=1} \frac{
\{t_i(\tau)\}^{z_i}
\{t_i(\tau+\Delta) - t_i(\tau)\}^{z'_i}
\lambda_i^{z_i + z_i' + \beta_i - 1}}{z_i!z'_i!} \ee^{-t_i(\tau + \Delta) \lambda_i} \\
& \times \int^\infty_0 u^{\alpha-1} \exp \biggl(- u \sum_j \frac{\lambda_j}{\gamma_j} \biggr) \dd u
\dd \lambda_1 \dotsb \dd \lambda_d \\
=& \int_0^\infty u^{\alpha-1} \int \prod^d_{i=1} \frac{
\{t_i(\tau)\}^{z_i}
\{t_i(\tau+\Delta) - t_i(\tau)\}^{z'_i}
\lambda_i^{z_i + z'_i  + \beta_i - 1}}{z_i! z'_i!}
\ee^{- \{\frac{u}{\gamma_i} + t_i(\tau+\Delta)\}\lambda_i}
\dd \lambda_1 \dotsb \dd \lambda_d \dd u \\
=& \int_0^\infty u^{\alpha-1} \prod^d_{i=1} 
 \frac{\{t_i(\tau)\}^{z_i} \{t_i(\tau+\Delta) - t_i(\tau)\}^{z'_i}
 \Gamma(z_i + z'_i + \beta_i)}
{z_i! z'_i! \{\frac{u}{\gamma_i} + t_i(\tau+\Delta)\}^{z_i + z'_i + \beta_i}} \dd u \\
=& \left[
\prod^d_{i=1}
\frac{\{t_i(\tau)\}^{z_i} \{t_i(\tau+\Delta) - t_i(\tau)\}^{z'_i}
\Gamma (z_i + z'_i + \beta_i)}{z_i! z'_i! \{t_i(\tau+\Delta)\}^{z_i + z'_i + \beta_i}} \right]
\int_0^\infty u^{\alpha-1} \prod^d_{j=1}
\frac{1}{\{\frac{u}{t_j(\tau+\Delta) \gamma_j} + 1\}^{z_j + z'_j + \beta_j}} \dd u,
\end{align*}
we obtain the desired result from
$p_{\alpha,\beta,\gamma} (z' \mid z) = p_{\alpha,\beta,\gamma}(z,z')/p_{\alpha,\beta,\gamma}(z)$.
\end{proof}

\vspace{0.5cm}

\begin{proof}[Proof of Lemma \ref{lemma-estimators}]
1)~
The posterior mean of $\lambda_i$ with respect to $\pi_\beta$ is given by
\begin{align*}
\hat{\lambda}_i^{(\beta)} &:= \frac{\int \lambda_i p(z(\tau) \mid \lambda) \pi_\beta(\lambda) \dd \lambda}
{\int p(z(\tau) \mid \lambda) \pi_\beta(\lambda) \dd \lambda}
= \frac{\displaystyle \int \lambda_i \prod^d_{j=1} \frac{\lambda_j^{z_j + \beta_i - 1}}{z_j!}
\ee^{-t_i(\tau) \lambda_j} \dd \lambda_1 \dotsb \dd \lambda_d}
{\displaystyle \int \prod^d_{k=1} \frac{\lambda_k^{z_k + \beta_k - 1}}{z_k!}
\ee^{-t_k(\tau) \lambda_k} \dd \lambda_1 \dotsb \dd \lambda_d} \\
&= \frac{\displaystyle \frac{\Gamma \left(z_i + \beta_i + 1 \right)}{t_i(\tau)^{z_i + \beta_i + 1}}
\prod_{j \neq i} \frac{\Gamma \left(z_j + \beta_i \right)}{t_k(\tau)^{z_k + \beta_i}}}
{\displaystyle \prod^d_{k=1} \frac{\Gamma(z_k + \beta_k)}{t_k(\tau)^{z_k + \beta_k}}}
= \frac{z_i + \beta_i}{t_i(\tau)}.
\end{align*}

\noindent
2)~
By using \eqref{priorintegral}, we have
\begin{align*}
\Gamma(\alpha) & \int p(z(\tau) \mid \lambda) \pi_{\alpha,\beta,\gamma}(\lambda) \dd \lambda \\
=& \int \prod^d_{i=1} \frac{\lambda_i^{z_i + \beta_i - 1}}{z_i!} \ee^{-t_i(\tau) \lambda_i}
\int^\infty_0 u^{\alpha-1} \exp 
\biggl(- u \sum_j \frac{\lambda_j}{\gamma_j} \biggr) \dd u
\dd \lambda_1 \dotsb \dd \lambda_d \\
=& \int_0^\infty u^{\alpha-1} \int \prod^d_{i=1} \frac{\lambda_i^{z_i + \beta_i - 1}}{z_i!}
\ee^{- \{t_i(\tau) + \frac{u}{\gamma_i}\}\lambda_i} 
\dd \lambda_1 \dotsb \dd \lambda_d \dd u \\
=& \left\{\prod^d_{i=1} \frac{\Gamma (z_i + \beta_i)}{z_i!} \right\} \int_0^\infty u^{\alpha-1}
\prod^d_{j=1} \frac{1}{\left\{t_j(\tau) + \frac{u}{\gamma_j}\right\}^{z_j + \beta_j}} \dd u,
\end{align*}
and
\begin{align*}
\Gamma(\alpha) & \int \lambda_i p(z(\tau) \mid \lambda) \pi_{\alpha,\beta,\gamma}(\lambda) \dd \lambda \\
=& \int \lambda_i \left( \prod^d_{j=1} \frac{\lambda_j^{z_j + \beta_j - 1}}{z_j!} \ee^{-t_j(\tau) \lambda_j} \right)
\int^\infty_0 u^{\alpha-1} \exp \biggl(- u \sum_k \frac{\lambda_k}{\gamma_k} \biggr) \dd u
\dd \lambda_1 \dotsb \dd \lambda_d \\
= & \frac{\Gamma \left(z_i + \beta_i + 1 \right)}{{z_i!}}
\left\{\prod_{j \neq i} \frac{\Gamma \left(z_j + \beta_j \right)}{{z_j!}} \right\} \\
& \times \int_0^\infty u^{\alpha-1}
\frac{1}{\left\{ t_i(\tau) + \frac{u}{\gamma_i} \right\}^{z_i + \beta_i + 1}}
\left[ \prod_{k \neq i} \frac{1}{\left\{ t_k(\tau) + \frac{u}{\gamma_k} \right\}^{z_k + \beta_k}} \right]
\dd u.
\end{align*}
Thus, the posterior mean of $\lambda$ with respect to $\pi_{\alpha,\beta,\gamma}$ is given by
\begin{align*}
\hat{\lambda}_i^{(\alpha,\beta,\gamma)} :=&
\frac{\displaystyle \int_0^\infty \lambda_i p(z(\tau) \mid \lambda) \pi_{\alpha,\beta,\gamma}(\lambda) \dd \lambda}
{\displaystyle \int_0^\infty p(z(\tau) \mid \lambda) \pi_{\alpha,\beta,\gamma}(\lambda) \dd \lambda}
= \frac{z_i + \beta_i}{t_i(\tau)}
\frac{\displaystyle
 \int_0^\infty \displaystyle
 u^{\alpha-1} \prod^d_{j=1} \dfrac{1}{\left\{\frac{u}{t_j(\tau) \gamma_j} + 1 \right\}^{z_j + \beta_j +\delta_{ij}}} \dd u}
{\displaystyle
 \int_0^\infty \displaystyle
u^{\alpha-1} \prod^d_{j=1} \dfrac{1}{\left\{\frac{u}{t_j(\tau) \gamma_j} + 1 \right\}^{z_j + \beta_j}} \dd u}.
\end{align*}
\end{proof}

\vspace{0.5cm}
\begin{proof}[Proof of Lemma \ref{lemma-koukan-0}]
We have
\begin{align*}
\E[x_i h(x) \mid \lambda] =&
\sum_x \prod^d_{j=1} \frac{\lambda_j^{x_j}}{x_j!} \ee^{-\lambda_j} x_i h (x)
= \sum_x \prod^d_{j=1} \lambda_i \frac{\lambda ^{x_j}_j}{x_j!} \ee^{-\lambda_j} h(x + \delta_i) \\
=& \E[\lambda_i h(x + \delta_i) \mid \lambda].
\end{align*}
\end{proof}

\begin{proof}[Proof of Lemma \ref{lemma-Ksekibun}]
1)~ By partial integration,
\begin{align*}
K (\gamma,&\ x, \alpha) = \int^\infty_0 u^{\alpha-1} \prod^d_{i=1} \frac{1}{(u/\gamma_i + 1)^{x_i}} \dd u \\
=& \left[ \frac{u^\alpha}{\alpha} \prod^d_{i=1} \frac{1}{(u/\gamma_i + 1)^{x_i}} \right]^\infty_0
+ \int^\infty_0 \frac{u^\alpha}{\alpha} \sum^d_{i=1} \left\{\prod_{j \neq i} \frac{1}{(u/\gamma_j + 1)^{x_j}} \right\}
\frac{x_i/\gamma_i}{(u/\gamma_i + 1)^{x_i +1}} \dd u \\
=& \frac{1}{\alpha} \sum_i \frac{x_i}{\gamma_i} K(\gamma, x + \delta_i, \alpha+1).
\end{align*}

\noindent
2)~
We have
\begin{align*}
K (\gamma, x + \delta_i, \alpha+1)
= & \int_0^\infty u^\alpha \Bigl\{\prod_j \frac{1}{(u/\gamma_j + 1)^{x_j}} \Bigr\} \frac{1}{u/\gamma_i + 1} \dd u \\
= & \int_0^\infty u^{\alpha-1} \Bigl\{\prod_j \frac{1}{(u/\gamma_j + 1)^{x_j}} \Bigr\} \frac{1}{u/\gamma_i + 1}
\gamma_i (u/\gamma_i + 1 - 1) \dd u \\
= & \gamma_i K (\gamma, x, \alpha) - \gamma_i K(\gamma, x + \delta_i, \alpha).
\end{align*}

\noindent
3)~
From \eqref{lemma-2-0}, we have
\begin{align*}
\sum_i b_i K (\gamma, x+ \delta_i, \alpha)
=& \sum_i b_i \left\{ K (\gamma, x, \alpha) - \frac{1}{\gamma_i} K (\gamma, x + \delta_i, \alpha+1) \right\}
 \notag \\
= & \frac{b_{\Cdot}}{\alpha} \alpha K(\gamma, x, \alpha) - \sum_i \frac{b_i}{\gamma_i} K (\gamma, x + \delta_i, \alpha+1).
\end{align*}
By using \eqref{lemma-1-0},
\begin{align*}
\sum_i b_i K (\gamma, x+ \delta_i, \alpha)
=&
\frac{b_{\Cdot}}{\alpha} \sum_i \frac{x_i}{\gamma_i} K (\gamma, x + \delta_i, \alpha+1)
- \sum_i \frac{b_i}{\gamma_i} K (\gamma, x + \delta_i, \alpha+1) \\
= & \sum_i \left(\frac{b_{\Cdot}}{\alpha} \frac{x_i}{\gamma_i} - \frac{b_i}{\gamma_i} \right) K(\gamma, x + \delta_i, \alpha+1) .
\end{align*}
\end{proof}

\section*{Acknowledgments}
This research was partially supported
by Grant-in-Aid for Scientific Research (23300104, 23650144)
and by the Aihara Project, the FIRST program from JSPS, initiated by CSTP.

\bibliographystyle{sjs}

\end{document}